\documentclass[letterpaper,11pt,oneside]{amsart}
\usepackage{amsmath}
\usepackage{amssymb}
\usepackage{latexsym}
\usepackage{amsfonts}
\usepackage[english]{babel}
\usepackage{multicol}
\usepackage{oldgerm}
\usepackage{enumerate}
\usepackage{hyperref}

\newtheorem{theorem}{Theorem}
\newtheorem{lemma}{Lemma}
\newtheorem{remark}{Remark}

\begin{document}

\title{On Brocard's problem with Padovan and Perrin numbers}

\author{Eric F. Bravo}
\address{Departamento de Matem\'aticas\\ Universidad del Cauca\\ Calle 5 No 4--70\\Popay\'an, Colombia.}
\email{fbravo@unicauca.edu.co}

\date{}

\begin{abstract}

The Padovan sequence $\{P_{m}\}_{m\ge 0}$ is a ternary recurrence sequence with companion polynomial $X^{3}-X-1$ and initial conditions $P_{0}=P_{1}=P_{2}=1$. The Perrin sequence $\{R_{m}\}_{m\ge 0}$ is defined by the same companion polynomial as the Padovan sequence, but has initial values $R_{0}=3$, $R_{1}=0$, and $R_{2}=2$. We solve the Brocard-Ramanujan equation $n!+1=x^{2}$, where $n!$ is the factorial of $n$ and $x$ is a Padovan number or a Perrin number. In both cases, we prove that $(n,x)=(4,5)$ is the only solution.

\medskip

\medskip



\end{abstract}

\maketitle


\section{Introduction}

\noindent Henri Brocard \cite{8} first posed the following problem in 1876: 
\begin{center}
``Find the integer values of $n$ for which $1\cdot 2\cdots n +1$ is a perfect square.''
\end{center}
Without knowing anything about Brocard's question, Srinivasa Ramanujan \cite{10} tackled the same problem in 1913:
\begin{center}
``The number $n!+1$ is a perfect square for the values $4,5,7$ of $n$. Find other values.''
\end{center}
We refer to the Diophantine equation
\begin{equation}\label{E3}
n!+1=x^{2}
\end{equation}
as the \textit{Brocard–Ramanujan equation}. The pairs $(n,x)$ that satisfy equation \eqref{E3} were named \textit{Brown numbers} by Pickover \cite{11}, after K. Brown brought the problem to his attention. To date, only three Brown numbers are known: $(4,5)$, $(5,11)$, and $(7,71)$. In 1980, Erd\H{o}s \cite{13} conjectured that these are the only solutions to equation \eqref{E3}. In 1993, Overholt \cite{14} proved that if the ABC conjecture is true, then there are only a finite number of solutions to the equation \eqref{E3}. Computational searches have not found any further solutions for $n\le 10^{9}$ (see Berndt and Galway \cite{9}). However, the Brocard-Ramanujan equation is one of the most famous Diophantine equations involving factorials that remains unsolved.

Among the numerous variants of Brocard's problem is the task of solving the equation \eqref{E3} when $x$ is a term in a linear recurrence sequence. The first to consider this variant was Marques \cite{4} in 2012, who proved that $(n,x)=(4,5)$ is the only solution to equation \eqref{E3} when $x=F_{m}$ is a Fibonacci number. To achieve this, Marques \cite{4} uses the Primitive Divisor Theorem of Carmichael \cite[Theorem XXI]{17} and a factorization of $F_{m}\pm 1$ as the product of a Fibonacci number and a Lucas number.

The first to solve the equation \eqref{E3} when $x$ is a term in a recurring ternary linear sequence were Facó and Marques \cite{7} in 2016, who proved that the Brocard-Ramanujan equation has no solutions when $x=T_{m}$ is a Tribonacci number.

To outline the approach of Facó and Marques \cite{7} for solving $n!+1=T_{m}^{2}$, recall that the $p$-adic valuation (or $p$-adic order) of an integer $x$ is defined as the highest power of the prime $p$ that divides $x$. We denote this by $\nu_{p}(x)$. Their idea basically consisted of showing that $\nu_{2}(T_{m}^{2}-1)$ grows only logarithmically in $m$. In relation to $n!=T_{m}^{2}-1$, this growth is much slower than the linear growth of $\nu_{2}(n!)$, which imposes an upper bound on $n$ for any solution.

Although the idea of Fac\'o and Marques \cite{7} idea was sound, it was based on an incorrect formula for $\nu_{2}(T_{m}+1)$ when $m\equiv 61\pmod{64}$. Recently, the reviewer of the article of Young \cite{16} pointed this out. Young \cite{1} used the argument of Fac\'o and Marques \cite{7}, but instead of using $\nu_{2}(T_{m}^{2}-1)$, he used $\nu_{47}(T_{m}^{2}-1)$ to correctly establish that there are no solutions to equation \eqref{E3} when $x=T_{m}$. Young \cite{16} also proved that there are no solutions to equation \eqref{E3} when $x$ is a Tetranacci number, this time using the $2$-adic order of shifted Tetranacci numbers.

Brocard's problem has also been solved using the idea of Fac\'o and Marques \cite{7} when $x$ is a Narayana number by means of the $3$-adic order of shifted Narayana numbers (see Ismail et al. \cite{3}).

In 2021, Bravo et al. \cite{2} proved, in general, that equation \eqref{E3} has a finite number of solutions when $x$ is a term of a linear recurrence sequence $\{u_{m}\}_{m\ge 0}$ that satisfies that $m=\mathcal{O}(\log u_{m})$, $\nu_{p}(u_{m}+1)=\mathcal{O}(m^{C_{1}})$, and $\nu_{p}(u_{m}-1)=\mathcal{O}(m^{C_{2}})$ for a prime $p$ and constants $C_{1},C_{2}$ such that $C=\max \{C_{1},C_{2}\}<1$, where $\mathcal{O}$ denotes the Landau symbol. They applied their result with $p=3$ to the Tripell sequence, proving that $(n,x)=(4,5)$ is the only Brown number for which $x$ is a Tripell number.

We consider Brocard's problem when $x$ is a Padovan or Perrin number. The Padovan numbers $\{P_{m}\}_{m\ge 0}$ are defined by the Fibonacci-like recurrence relation
$$
P_{m+3}=P_{m+1}+P_{m}\quad \text{for}\quad m\ge 0,
$$
with initial conditions $P_{0}=P_{1}=P_{2}=1$. The first few are
$$
1,1,1,2,2,3,4,5,7,9,12,16,21,28,37,49,65,86,\ldots.
$$ 
The Perrin numbers $\{R_{m}\}_{m\ge 0}$ are defined by the same recurrence relation as Padovan numbers, but with initial values $R_{0}=3$, $R_{1}=0$, and $R_{2}=2$. The first few are
$$
3,0,2,3,2,5,5,7,10,12,17,22,29,39,51,68,90,\ldots.
$$

\section{Results}

\noindent Our main findings are as follows.
\begin{theorem}\label{T1}
The pair $(n,x)=(4,5)$ is the only Brown number for which $x$ is a Padovan number.
\end{theorem}
\begin{theorem}\label{T2}
There is only one solution to the Brocard-Ramanujan equation when $x$ is a Perrin number, namely $(n,x)=(4,5)$.
\end{theorem}
We use the method of Facó and Marques \cite{7} to prove both theorems. In the case where $x$ is a Padovan number and $m\not \equiv 50\pmod{112}$, we use the $2$-adic valuation of shifted Padovan numbers, whose formulas were recently discovered by Batte et al. \cite[Proposition 2.12]{6}, and Bravo and Irmak \cite[Theorem 2.1]{5}. If $m\equiv 50\pmod{112}$, we use the $7$-adic valuation of the shifted Padovan numbers, whose formulas we derive here (see Lemmas \ref{L2} and \ref{L1}). The growth of the $p$-adic valuation of $n!$ is linear regardless of the value of $p$, so in any case we obtain an upper bound for $n$. This upper bound is so small that there is no need to use a computational routine to find the solutions to equation \eqref{E3}; it is sufficient to check whether $x=5$, $11$, or $71$ are Padovan numbers, based on an earlier result by Gupta \cite{12}.

The case where $x=R_{m}$ is a Perrin number is similar to the case where $x$ is a Padovan number, with the difference that there are two special classes, namely $m\equiv 5,10\pmod{14}$.

\section{Auxiliary Results}

\noindent We begin with a complete description of the $2$-adic valuation of the shifted Padovan sequence $\{P_{m}-1\}_{m\ge 0}$ given by Bravo and Irmak \cite[Theorem 2.1]{5}.
\begin{theorem}\label{TV1}
For $m\ge 0$, we have
$$
\nu_{2}\left(P_{m}-1\right)= 
\begin{cases}
\infty, & \mbox{if } m=0,1,2;\\
0, & \mbox{if } m\equiv 3,4,6\pmod{7};\\
\nu_{2}(m+2)+1, & \mbox{if }  m\equiv 5\pmod{7};\\
\nu_{2}((m-1)(m+13))+1, & \mbox{if }  m\equiv 1\pmod{7};\\
\nu_{2}(m)+1, & \mbox{if } m\equiv 0\pmod{14};\\
\nu_{2}(m+7)+1, & \mbox{if }  m\equiv 7\pmod{14};\\
\nu_{2}(m+5)+2, & \mbox{if }  m\equiv 9\pmod{14};\\
\nu_{2}((m-2)(m+26))+3, & \mbox{if }  m\equiv 2\pmod{28};\\
\nu_{2}(m+12)+4, & \mbox{if }  m\equiv 16\pmod{28}.
\end{cases}
$$
\end{theorem}
We continue with the $2$-adic order of the shifted Padovan sequence $\{P_{m}+1\}_{m\ge 0}$ except when $m\equiv 50\pmod{112}$. This result is due to Batte, Bravo, and Luca \cite[Proposition 2.12]{6}.
\begin{theorem}\label{P1}
For all $m\ge 0$ such that $m\not\equiv 50\pmod{112}$, we have
$$
\nu_{2}(P_{m}+1)= 
\begin{cases}
0, & \mbox{if } m\equiv 3,4,6\pmod{7};\\
\nu_{2}(m+9)+1, & \mbox{if }  m\equiv 5\pmod{7};\\
1, & \mbox{if }  m\equiv 0,1,2,7,9\pmod{14};\\
3, & \mbox{if }  m\equiv 8\pmod{28};\\
5, & \mbox{if }  m\equiv 22\pmod{56};\\
\nu_{2}(m+6)+4, & \mbox{if }  m\equiv 106\pmod{112}.
\end{cases}
$$
\end{theorem}
In the following two results, we determine the $7$-adic valuation of the shifted Padovan sequences $\{P_{m}\pm 1\}_{m\ge 0}$ for the case where $m\equiv 50\pmod{112}$.
\begin{lemma}\label{L2}
If $m\equiv 50\pmod{112}$, then $\nu_{7}(P_{m}+1)=0$.
\end{lemma}
\begin{proof}
The companion polynomial of $\{P_{m}\}_{m\ge 0}$ is $P(X)=X^{3}-X-1$. This polynomial factors into $(X-5)(X^{2}+5X+3)$ in the finite field $\mathbb{F}_{7}$, since $P(5)=119\equiv 0\pmod{7}$ and $X^{2}+5X+3$ has no real roots in $\mathbb{F}_{7}$ because its discriminant, namely $5^{2}-4(1)(3)=13$, is not a perfect square modulo $7$. The order of $5$ in $\mathbb{F}_{7}^{\ast}$ is $6$. The complex conjugate roots $\beta$ and $\gamma$ of $X^{2}+5X+3$ are primitive elements of $\mathbb{F}^{\ast}_{49}$, that is, they have order $48$. Therefore, the period of the Padovan sequence modulo $7$ is the least common multiple of $6$, $48$, and $48$, namely $48$. 

Now, let $k$ be a non-negative integer such that $m=112k+50$. Then, $m\equiv 16k+2\pmod{48}$. If $k\equiv 0,1,2\pmod{3}$, then $m\equiv 2,18,34\pmod{48}$, respectively. Since $P_{2}=1\equiv 1\pmod{7}$, $P_{18}=114\equiv 2\pmod{7}$, and $P_{34}=10252\equiv 4\pmod{7}$, we conclude that $P_{m}\pmod{7}\in \{1,2,4\}$. Therefore, $P_{m}+1\pmod{7}\in \{2,3,5\}$ and, consequently, $7$ does not divide $P_{m}+1$.
\end{proof}
\begin{lemma}\label{L1}
Suppose that $m\equiv 50\pmod{112}$. Then
$$
\nu_{7}(P_{m}-1)=
\begin{cases}
0, & \mbox{if }  m\equiv 162,274\pmod{336};\\
1, & \mbox{if }  m\equiv 50\pmod{336}.\\
\end{cases}
$$
\end{lemma}
\begin{proof}
Let $m\equiv 162,274\pmod{336}$. Then there exist non-negative integers $k$ and $\ell$ such that $m=336k+162$ or $m=336\ell +274$. Therefore, $m\equiv 18,34\pmod{48}$. Since the period of the Padovan sequence modulo $7$ is $48$, $P_{18}\equiv 2\pmod{7}$, and $P_{34}\equiv 4\pmod{7}$, we conclude that $P_{m}\pmod{7}\in \{2,4\}$. Therefore, $P_{m}-1\pmod{7}\in \{1,3\}$ and, consequently, $7$ does not divide $P_{m}-1$.

Now, let $m\equiv 50\pmod{336}$. The period of the Padovan sequence modulo $49$ is $336$. Therefore, 
$$
P_{m}\equiv P_{50}\pmod{49}.
$$
Since $P_{50}=922111\equiv 29\pmod{49}$, it follows that $P_{m}-1\equiv 28\pmod{49}$ for any $m\equiv 50\pmod{336}$. This proves that $7$ divides $P_{m}-1$ but $7^{2}=49$ does not.
\end{proof}
Below we present the $2$-adic valuation of the shifted Perrin sequence $\{R_m+1\}_{m\ge 0}$, which was fully characterized by Bravo and Irmak \cite[Theorem 2.2]{5}.
\begin{theorem}\label{TV2}
For $m\ge 0$, we have
$$
\nu_{2}\left(R_{m}+1\right)= 
\begin{cases}
0, & \mbox{if } m\equiv 1,2,4\pmod{7};\\
1, & \mbox{if }  m\equiv 5\pmod{7};\\
\nu_{2}(m+7)+2, & \mbox{if }  m\equiv 0\pmod{7};\\
\nu_{2}(m+11)+1, & \mbox{if } m\equiv 3\pmod{7};\\
\nu_{2}((m+1)(m+29))+1, & \mbox{if }  m\equiv 6\pmod{7}.
\end{cases}
$$
\end{theorem}
The $2$-adic order of the shifted Perrin sequence $\{R_{m}-1\}_{m\ge 0}$ for all $m\not\equiv 5,10\pmod{14}$, which we present below, can be found in Batte, Bravo, and Luca \cite[Proposition 2.14]{6}.
\begin{theorem}\label{TV3}
For any $m\ge 0$ such that $m\not\equiv 5,10\pmod{14}$, we have
$$
\nu_{2}(R_{m}-1)= 
\begin{cases}
0, & \mbox{if } m\equiv 1,2,4\pmod{7};\\
1, & \mbox{if } m\equiv 0,3,7,13\pmod{14};\\
2, & \mbox{if } m\equiv 6\pmod{14};\\
\nu_{2}(m+2)+1, & \mbox{if } m\equiv 12\pmod{14}.
\end{cases}
$$
\end{theorem}
For classes $m\equiv 5,10\pmod{14}$, we determine the $7$-adic valuation of the shifted Perrin sequences $\{R_{m}\pm 1\}_{m\ge 0}$ in the following two results.
\begin{lemma}\label{L6}
If $m\equiv 5\pmod{14}$ or $m\equiv 10\pmod{14}$, then
$$
\nu_{7}(R_{m}-1)= 
\begin{cases}
0, & \mbox{if } m\not \equiv 94,108,173,178,276,299\pmod{336};\\
1, & \mbox{otherwise. } 
\end{cases}
$$
\end{lemma}
\begin{proof}
Suppose that $m\equiv 5\pmod{14}$. Then, $m=14k+5$ for some integer $k\ge 0$. Therefore, $m\equiv 14k+5\pmod{48}$. Since the companion polynomial of the Perrin sequence is the same as that of the Padovan sequence, it follows that the Perrin sequence also has a period of $48$ modulo $7$. Therefore,
$$
R_{m}\equiv R_{14k+5}\pmod{7}.
$$
It can be shown that $R_{14k+5}\equiv 1\pmod{7}$ if $k\equiv 12,21\pmod{24}$. Otherwise, $R_{14k+5}\not \equiv 1\pmod{7}$. Therefore, $\nu_{7}(R_{m}-1)=0$ if $m\not \equiv 173,299\pmod{336}$. 

If $k\equiv 12,21\pmod{24}$, then $m\equiv 173,299\pmod{336}$. But $R_{173}\equiv 15\pmod{49}$, $R_{299}\equiv 22\pmod{49}$, and the period of the Perrin sequence modulo $49$ is $336$, so
$$
R_{m}\equiv 15,22\pmod{49}.
$$
Hence, $R_{m}-1\equiv 14,21\pmod{49}$. Thus, $\nu_{7}(R_{m}-1)=1$ if $m\equiv 173,299\pmod{336}$.

Suppose now that $m\equiv 10\pmod{14}$. Let $\ell \ge 0$ be an integer such that $m=14\ell +10$. Therefore, $m\equiv 14\ell +10\pmod{48}$. This implies that
$$
R_{m}\equiv R_{14\ell +10}\pmod{7},
$$
since $\{R_{m}\}$ has a period of $48$ modulo $7$. It can be seen that $R_{14\ell +10}\equiv 1\pmod{7}$ if $\ell \equiv 6,7,12,19\pmod{24}$. Otherwise, $R_{14\ell +10}\not \equiv 1\pmod{7}$. Therefore, $\nu_{7}(R_{m}-1)=0$ if $m\not \equiv 94,108,178,276\pmod{336}$. 

If $\ell \equiv 6,7,12,19\pmod{24}$, then $m\equiv 94,108,178,276\pmod{336}$. Since $R_{94}\equiv R_{276}\equiv 15\pmod{49}$, $R_{108}\equiv 8\pmod{49}$, $R_{178}\equiv 36\pmod{49}$, and the period of the Perrin sequence modulo $49$ is $336$, it follows that
$$
R_{m}\equiv 8,15,36\pmod{49}.
$$
Therefore, $R_{m}-1\equiv 7,14,35\pmod{49}$, which means that $\nu_{7}(R_{m}-1)=1$ if $m\equiv 94,108,178,276\pmod{336}$.
\end{proof}
\begin{lemma}\label{L5}
If $m\equiv 5\pmod{14}$ or $m\equiv 10\pmod{14}$, then
$$
\nu_{7}(R_{m}+1)= 
\begin{cases}
0, & \mbox{if } m\not \equiv 19,24,38,47,89,122,229,304\pmod{336};\\
1, & \mbox{if } m\equiv 19,24,38,89,122,229,304\pmod{336};\\
2, & \mbox{if } m\equiv 47\pmod{336}.
\end{cases}
$$
\end{lemma}
\begin{proof}
The proof of the first two cases is similar to that of Lemma \ref{L6}. For this reason, we omit the details of their proofs. For the case $m\equiv 47\pmod{336}$, note that $R_{47}=549289\equiv 48\pmod{49}$, so $\nu_{7}(R_{m}+1)\ge 2$, since the period of the Perrin sequence modulo $49$ is $336$. To prove that $\nu_{2}(R_{m}+1)$ is exactly $2$, it it suffices to show that $R_{m}+1$ is not divisible by $7^{3}=343$. This holds because $R_{m}+1\equiv 147\pmod{343}$ for all $m\equiv 47\pmod{336}$.
\end{proof}
We now provide an exponential lower bound for $P_{m}$ in terms of the real root $\alpha$ of $P(X)$ and its index $m$ (see Batte, Bravo, and Luca \cite[Lemma 2.3]{6} for a proof.)
\begin{lemma}\label{L4}
For $m\ge 0$, we have
$$
P_{m}\ge \alpha^{m-2}.
$$
\end{lemma}
A similar result for Perrin numbers is proven in Bravo and Irmak \cite[Equation (8.3)]{5}.
\begin{lemma}\label{L7}
For $m\ge 2$, we have
$$
R_{m}>\alpha^{m-2}.
$$
\end{lemma}
We conclude this section on auxiliary results with a lower bound for the $p$-adic order of $n!$. A proof of this result can be found in Bugeaud and Laurent \cite[Lemma 1]{15}.
\begin{lemma}\label{L3}
For any positive integer $n$, we have
$$
\nu_{p}(n!)\ge \frac{n}{p-1}-\log_{p}(n+1).
$$
\end{lemma}

\section{On the Brocard–Ramanujan equation with Padovan numbers}

\noindent In this section, we prove Theorem \ref{T1}. Let $(n,x)$ be a solution of the Brocard-Ramanujan equation with $n\ge 4$ and $x\ge 5$, where $x$ is a Padovan number. Then exists an integer $m\ge 7$ such that $x=P_{m}$. Therefore,
\begin{equation}\label{E1}
n!=(P_{m}+1)(P_{m}-1).
\end{equation}
\textsc{Case 1.} $m\equiv 50\pmod{122}$.

By applying the $7$-adic valuation to both sides of \eqref{E1} and using the fact that $\nu_{p}(r\cdot s) =\nu_{p}(r)+\nu_{p}(s)$ for any $r,s\in \mathbb{Z}$ and $p$ a prime number, we obtain
\begin{equation*}
\nu_{7}(n!)=\nu_{7}(P_{m}+1)+\nu_{7}(P_{m}-1).
\end{equation*}
The left-hand side is greater than or equal to $(n/6)-\log_{7}(n+1)$ by Lemma \ref{L3} with $p=7$, while the right-hand side is less than or equal to $1$ by Lemmas \ref{L2} and \ref{L1}. Therefore,
$$
(n/6)-\log_{7}(n+1)\le 1,
$$
and so
\begin{equation*}
n\le 6\log_{7}(n+1)+6.
\end{equation*}
Since the right-hand side of the above inequality is concave down as a function of real $n > 1$, a routine calculation shows that this inequality requires that
$$
n\le 14.
$$

\noindent \textsc{Case 2.} $m\not \equiv 50\pmod{112}$.

Now we apply the $2$-adic valuation to both sides of \eqref{E1} and again use the homomorphism property of the valuation, which shows that
\begin{equation*}
\nu_{2}(n!)=\nu_{2}(P_{m}+1)+\nu_{2}(P_{m}-1).
\end{equation*}
By Lemma \ref{L3} with $p=2$, we obtain that $\nu_{2}(n!)\ge n-\log_{2}(n+1)$. Therefore,
\begin{equation}\label{E6}
n\le \log_{2}(n+1)+\nu_{2}(P_{m}+1)+\nu_{2}(P_{m}-1).
\end{equation}
From Theorem \ref{P1} and the fact that $\nu_{2}(r)\le \log_{2}(r)$ for $r\in \mathbb{Z}^{+}$, it follows that 
\begin{align}\label{E4}
\nonumber
\nu_{2}(P_{m}+1)&\le \max\{\nu_{2}(m+9)+1,\nu_{2}(m+6)+4,5\}\\
&\le \nu_{2}(m+6)+4\le \log_{2}(m+6)+4.
\end{align}
Meanwhile, Theorem \ref{TV1} implies that 
\begin{align}\label{E5}
\nonumber
\nu_{2}(P_{m}-1)&\le \max \{\nu_{2}(m+2)+1,\nu_{2}((m-1)(m+13))+1,\nu_{2}(m+7)+1,\\\nonumber
&\nu_{2}(m)+1,\nu_{2}(m+5)+2,\nu_{2}((m-2)(m+26))+3,\nu_{2}(m+12)+4\}\\
&\le \nu_{2}((m-2)(m+26))+3\le \log_{2}((m-2)(m+26))+3.
\end{align}
Therefore, from \eqref{E6}, \eqref{E4}, and \eqref{E5}, we conclude that
\begin{equation}\label{E8}
n\le \log_{2}((n+1)(m+6)(m-2)(m+26))+7.
\end{equation}
On the other hand, by the definition of the factorial, we have $n!<n^{n}$ for all $n\ge 2$. Hence, $n!+1\le n^{n}$ for all $n\ge 2$. Since $P_{m}^{2}\ge \alpha^{2m-4}$ by Lemma \ref{L4}, it follows that 
$$
\alpha^{2m-4}\le P_{m}^{2}=n!+1\le n^{n},
$$
and therefore
\begin{equation}\label{E7}
m\le 2n\log_{2}(n)+2.
\end{equation}
From \eqref{E8} and \eqref{E7} we obtain that
$$
n\le \log_{2}\left((n+1)(n\log_{2}(n)+4)(n\log_{2}(n))(n\log_{2}(n)+14)\right)+10.
$$
Since the right-hand side of the above inequality is concave down as a function of real $n>1$, a routine calculation shows that
$$
n\le 38.
$$
Therefore, in any case, $n\le 38$. In 1935, Gupta \cite{12} stated that calculations of $n!$ up to $n=63$ yielded no solutions to the Brocard–Ramanujan equation other than $(n,x)\in \{(4,5),(5,11),(7,71)\}$. The proof of Theorem \ref{T1} concludes by noting that $x=5=P_{7}$ is a Padovan number, whereas $x=11$ and $x=71$ are not.

\section{Perrin sequence and the Brocard–Ramanujan equation}

\noindent We conclude by proving Theorem \ref{T2}.  Let $(n,x)$ be a pair of Brown numbers with $n\ge 4$ and $x\ge 5$, where $x$ is a Perrin number. Then exists an integer $m\ge 5$ such that $x=R_{m}$. Therefore, 
\begin{equation}\label{E2}
n!=(R_{m}+1)(R_{m}-1).
\end{equation}

\noindent \textsc{Case 1.} $m\equiv 5\pmod{14}$ or $m\equiv 10\pmod{14}$.

Here we apply the $7$-adic valuation to \eqref{E2} and use the fact that $\nu_{p}(r\cdot s)=\nu_{p}(r)+\nu_{p}(s)$ for any $r,s\in \mathbb{Z}$ and $p$ a prime number. We then apply Lemma \ref{L3} (with $p=7$), Lemma \ref{L5}, and Lemma \ref{L6} to the resulting equality. This gives us
$$
(n/6)-\log_{7}(n+1)\le \nu_{7}(n!)=\nu_{7}(R_{m}+1)+\nu_{7}(R_{m}-1)\le 2+1.
$$
Therefore,
$$
n\le 6\log_{7}(n+1)+18.
$$
Since the right-hand side of the above inequality is concave down as a function of real $n > 1$, a routine calculation shows that this inequality implies that
$$
n\le 28.
$$

\noindent \textsc{Case 2.} $m\not \equiv 5,10\pmod{14}$.

Taking the $2$-adic valuation in \eqref{E2} and using the fact that the $p$-adic valuation of a product is equal to the sum of the $p$-adic valuations of the factors, we obtain
$$
\nu_{2}(n!)=\nu_{2}(R_{m}+1)+\nu_{2}(R_{m}-1).
$$
We know that $\nu_{2}(n!)\ge n-\log_{2}(n+1)$ by Lemma \ref{L3} with $p=2$. Meanwhile, $\nu_{2}(R_{m}+1)\le \nu_{2}((m+1)(m+29))+1$ and $\nu_{2}(R_{m}-1)\le \nu_{2}(m+2)+1$ by Theorems \ref{TV2} and \ref{TV3}, respectively. Therefore,
$$
n\le \log_{2}(n+1)+\nu_{2}((m+1)(m+29))+\nu_{2}(m+2)+2.
$$
Since $\nu_{2}(r)\le \log_{2}(r)$ for any positive integer $r$, we must have
\begin{equation}\label{E11}
n\le \log_{2}(4(n+1)(m+1)(m+29)(m+2)).
\end{equation}
Furthermore, it follows from Lemma \ref{L7} that $R_{m}>\alpha^{2m-4}$. Since $n!+1\le n^{n}$ for all $n\ge 2$, we obtain
$$
\alpha^{2m-4}<R_{m}^{2}=n!+1\le n^{n},
$$
and therefore
\begin{equation}\label{E10}
m<2n\log_{2}(n)+2.
\end{equation}
From \eqref{E11} and \eqref{E10}, we can conclude that
\begin{equation*}
n<\log_{2}\left(8(n+1)(2n\log_{2}(n)+3)(2n\log_{2}(n)+31)(n\log_{2}(n)+2)\right).
\end{equation*}
Since the right-hand side of the above inequality is concave down as a function of real $n>1$, a routine calculation shows that
$$
n\le 32.
$$
Therefore, in any case, $n\le 32$. Since $n<63$, the result of Gupta \cite{12} guarantees that the only solutions to the Brocard-Ramanujan equation are the three known Brown numbers. However, $x=11$ and $x=71$ are not Perrin numbers, whereas $x=5=R_{5}=R_{6}$ is. This concludes the proof of Theorem \ref{T2}.

\begin{remark}
It follows from Lemma \ref{L4} that $m<7\log P_{m}$ for all $m\ge 3$; therefore, $m = \mathcal{O}(\log P_{m})$. On the other hand, for $m\not \equiv 50\pmod{112}$, it follows from \eqref{E4} and \eqref{E5}, respectively, that
$$
\nu_{2}(P_{m}+1)\le \log_{2}(m+6)+4<5m^{1/3},
$$
and
$$
\nu_{2}(P_{m}-1)\le \log_{2}((m-2)(m+26))+3<5.5m^{1/3}
$$
hold for all $m\ge 3$. Thus, $\nu_{2}(P_{m}+1)=\mathcal{O}(m^{1/3})$ and $\nu_{2}(P_{m}-1)=\mathcal{O}(m^{1/3})$. It follows from Bravo et al. \cite[Theorem 1]{2} that equation \eqref{E3} has a finite number of solutions when $x=P_{m}$ is a Padovan number and $m\not \equiv 50\pmod{112}$. In the proof of Bravo et al. \cite[Theorem 1]{2}, we can take the parameters $(p,n_{0},K,K_{3},C)=(2,6,5.5,7,1/3)$, and so
$$
\hat{K}=2/(K_{3}(2Kp)^{1/C})=0.0000268327\ldots.
$$
By using this, from Bravo et al. \cite[Equation (10)]{2} we obtain the inequality
$$
n^{2}<37268\log (n/2)
$$
implying that $n\le 449$. In the class $m\equiv 50\pmod{112}$, the $2$-adic valuation of $\{P_{m}+1\}_{m\ge 0}$ does not grow logarithmically in $m$, and there is no constant $C_{1}$ such that $\nu_{2}(P_{m}+1)=\mathcal{O}(m^{1/C_{1}})$. Therefore, it is not possible to apply Theorem 1 of Bravo et al. \cite{2} with $p=2$ to conclude that there are finitely many solutions to equation \eqref{E3} when $x=P_{m}$ is a Padovan number and $m\equiv 50\pmod{112}$.

If $p=7$ and $m\equiv 50\pmod{112}$, then according to Lemmas \ref{L2} and \ref{L1}, we cannot apply Theorem 1 of Bravo et al. \cite{2} either, since $\nu_{7}(P_{m}+1)=\mathcal{O}(1)$ and $\nu_{7}(P_{m}-1) =\mathcal{O}(1)$, which implies that $C=0$.
\end{remark}

\begin{remark}
Nor is it possible to use Theorem 1 of Bravo et al. \cite{2} with $p=2$ and $p=7$ to conclude that there are finitely many solutions to the Brocard-Ramanujan equation when $x=R_{m}$ is a Perrin number and $m\equiv 5,10\pmod{14}$.
\end{remark}




\begin{thebibliography}{99}

\bibitem{6} Batte, H., Bravo, E. F., Luca, F.: Cullen and Woodall numbers in Padovan and Perrin sequences. arXiv:2605.23084 [math.NT]. (2026). 
https://doi.org/10.48550/arXiv.2605.23084

\bibitem{9} Berndt, B. C., Galway, W. F.: On the Brocard-Ramanujan equation $n!+1=m^{2}$. Ramanujan J. {\bf{4}} (2), 41-42 (2000).

\bibitem{5} Bravo, E. F., Irmak, N.: The 2-adic valuation of shifted Padovan and Perrin numbers and applications. Turkish J. Math. {\bf{48}} (6), 1183-1196 (2024) https://www.doi.org/10.55730/1300-0098.3568

\bibitem{2} Bravo J. J., Diaz, M., Ramirez, J. L.: On a variant of the Brocard-Ramanujan equation and an application. Publ. Math. Debrecen {\bf{98}} (1-2), 243-253 (2021) https://www.doi.org/10.5486/PMD.2021.8871

\bibitem{8} Brocard, H.: Question 166.  Nouv. Corres. Math. {\bf 2} 287-287, (1876).

\bibitem{15} Bugeaud, Y., Laurent, M.: Minoration effective de la distance p-adique entre puissances de nombres algébriques. J. Number Theory {\bf{61}}, 311–342 (1996).

\bibitem{17} Carmichael, R. D.: On the numerical factors of the arithmetic forms $\alpha^{n}\pm \beta^{n}$. Ann. of Math. {\bf{15}}, 30-70 (1913).

\bibitem{13} Erd\H{o}s, P., Graham, R.: Old and New Problems and Results in Combinatorial Number Theory. Vol. 28 of Monographies de L'Enseignement Mathématique. Geneva: Université de Genève, 1980, p. 97.

\bibitem{7} Facó, V., Marques, D.: Tribonacci Numbers and the Brocard-Ramanujan Equation. J. Integer Seq. {\bf{19}}, A 16.4.4 (2016).

\bibitem{12} Gupta, H.: On a Brocard-Ramanujan Problem. Math. Student {\bf{3}} 71--71, (1935).

\bibitem{3} Ismail, M., Rihane, S. E., Anwar, M.: Narayana sequence and the Brocard-Ramanujan equation. Notes on Number Theory and Discrete Mathematics {\bf{29}} (3), 462-473 (2023) https://www.doi.org/10.7546/nntdm.2023.29.3.462-473

\bibitem{4} Marques, D.: Fibonacci numbers at most one away from a product of factorials. Notes on Number Theory and Discrete Mathematics {\bf{18}} (3), 13-19 (2012).

\bibitem{14} Overholt, M.: The Diophantine Equation $n! + 1 = m^{2}$. Bull. Lond. Math. Soc. {\bf{25}} (2): 104--104, (1993). https://www.doi.org/10.1112/blms/25.2.104

\bibitem{11} Pickover, C. A.: Keys to Infinity, John Wiley \& Sons, p. 170 (1995).

\bibitem{10} Ramanujan, S.: Collected Papers of Srinivasa Ramanujan (Ed. G. H. Hardy, P. V. S. Aiyar, and B. M. Wilson). Providence, RI: Amer. Math. Soc., p. 327, 2000.

\bibitem{16} Young, P.T.: 2-adic properties of generalized Fibonacci numbers. Integers {\bf{20}}, A71 (2020). https://www.doi.org/10.5281/zenodo.10792544

\bibitem{1} Young, P.T.: On the Brocard-Ramanujan equation with Tribonacci and Tetranacci numbers. Integers {\bf{24}}, A92 (2024). https://www.doi.org/10.5281/zenodo.13992577

\end{thebibliography}
\end{document}